\documentclass[reqno, 12pt]{amsart} 
\usepackage{amsmath, amssymb} %% AMS symbol definitions
\usepackage{amscd}            %% for Commutative Diagrams
\usepackage{amsxtra}          %% defines \spcheck, \spdot, etc.
\usepackage{upref}            %% cross-refs to sections are upright
\usepackage{mathrsfs}
\theoremstyle{plain}

\newtheorem{theorem}{Theorem}
\newtheorem{proposition}[theorem]{Proposition}
\newtheorem{lemma}[theorem]{Lemma}
\newtheorem{corollary}[theorem]{Corollary}

\theoremstyle{remark}
\newtheorem{remark}[theorem]{Remark}

\newcommand{\biborder}[1]{}
\numberwithin{equation}{section}
\numberwithin{theorem}{section}
\newcommand{\be}%
  {\protect\setcounter{equation}{\value{subsubsection}}}  
\newcommand{\ee}%
  {\protect\setcounter{subsubsection}{\value{equation}}}

\DeclareMathAlphabet\BOONDOX{U}{rsfso}{m}{n}
\newcommand{\Mod}[1]{\ (\mathrm{mod}\ #1)}

\newcommand{\N}{{\mathbb N}}

\newcommand{\Q}{{\mathbb Q}}
\newcommand{\R}{{\mathbb R}}
\newcommand{\C}{{\mathbb C}}
\newcommand{\sel}{\BOONDOX{S}}

\newcommand{\bselp}{\BOONDOX{L}}

\newcommand{\esel}{{\sel}^{\#}}

\newcommand{\gselp}{\mathfrak{G}}
\newcommand{\gsel}{\gselp^{\#}}

\newcommand{\lselp}{\mathfrak{A}}
\newcommand{\lsel}{\lselp^{\#}}
\newcommand{\lseld}{\lsel_d}

\newcommand{\eff}{\mathcal F}

\newcommand{\aq}{{\mathbb A}_{\mathbb Q}}
\newcommand{\ak}{{\mathbb A}_K}
\newcommand{\glone}{{\rm GL}_1}
\newcommand{\gltwo}{{\rm GL}_2}

\newcommand{\glnaq}{{\rm GL}_n({\mathbb A}_{\mathbb Q})}

\newcommand{\glnak}{{\rm GL}_n(\ak)}
\newcommand{\glniak}{{\rm GL}_{n_i}(\ak)}

\newcommand{\res}{\rm{Re}(s)}
\newcommand{\re}{\text{Re}}

\begin{document}

%%%%%%%%%%%%%%%%%%%%%%%%%%%%%%%%%%%%%%%%%%%%%%%%%%%%%%%%%%%%%%%%%%%%%
%%%%%%%%%%%%%%%%%%%%%%%%%%%%%%%%%%%%%%%%%%%%%%%%%%%%%%%%%%%%%%%%%%%%%
%% TOPMATTER
%%%%%%%%%%%%%%%%%%%%%%%%%%%%%%%%%%%%%%%%%%%%%%%%%%%%%%%%%%%%%%%%%%%%%

\title{Beyond the extended Selberg class: $d_F\le 1$}
%%%%%
%%%%% Version of R. Raghunathan, August 06, 2017.
%%%%%
\author{Ravi ~Raghunathan}

\address{Department of Mathematics,
         Indian Institute of Technology Bombay,
         Powai, Mumbai,\enspace  400076, India}
\email{ravir@math.iitb.ac.in}
\subjclass[2010]{11M41,11F66}
\keywords{Selberg class, automorphic $L$-functions, converse theorems}

\begin{abstract} 
We will introduce two new classes of Dirichlet series which are
monoids under multiplication. The first class $\lsel$ contains
both the extended Selberg class $\esel$ of Kaczorowski and
Perelli as well as many
$L$-functions attached to automorphic representations of $\glnak$,
where $\ak$ denotes the ad\`eles over the number field $K$
(these representations need not be unitary or generic).
This is in contrast to the class $\esel$ which is smaller and is known to contain, very
few of these $L$-functions. 
The larger class is obtained by weakening
the requirement for absolute convergence, 
allowing a finite number of poles,
allowing more general gamma factors and by allowing 
the series to have trivial zeros to the right of $\res=1/2$, while retaining
the other axioms of the extended Selberg class.
We will classify series in $\lsel$ of degree $d$ when $d\le 1$
(when $d=1$, we will assume absolute convergence in $\res>1$).
We will further prove a primitivity result for the $L$-functions
of cuspidal eigenforms on $\gltwo(\aq)$ and a theorem allowing us to compare the zeros
of tensor product $L$-functions of $\glnak$ which cannot be deduced from 
previous classification results. The second class
$\gsel\subset\lsel$, which also contains $\esel$, more closely models the behaviour of 
$L$-functions of unitary globally generic representations of $\glnak$.

\end{abstract}
\vskip 0.2cm

\maketitle

%%    LEFT AND RIGHT RUNNING HEADS.  MUST COME *AFTER* \maketitle

\markboth{RAVI RAGHUNATHAN}{Classification of degree one elements}

%%%%%%%%%%%%%%%%%%%%%%%%%%%%%%%%%%%%%%%%%%%%%%%%%%%%%%%%%%%%%%%%%%%%%
%% BODY OF PAPER BEGINS HERE
%%%%%%%%%%%%%%%%%%%%%%%%%%%%%%%%%%%%%%%%%%%%%%%%%%%%%%%%%%%%%%%%%%%%%
%%%%%%%%%%%%%%%%%%%%%%%%%%%%%%%%%%%%%%%%%%%%%%%%%%%%%%%%%%%%%%%%%%%%%
%% INTRODUCTION
%%%%%%%%%%%%%%%%%%%%%%%%%%%%%%%%%%%%%%%%%%%%%%%%%%%%%%%%%%%%%%%%%%%%%

\section{Introduction}\label{introduction}
The purpose of this paper is to introduce two classes of Dirichlet series 
$\lsel$ and $\gsel$ and their arithmetic counterparts $\lselp$ and
$\gselp$ which we believe provide the correct setting for the study of the
analytic theory of automorphic $L$-functions. These classes are obtained
by weakening the hypotheses used to define the 
extended Selberg class $\esel$ studied by Kaczorowski and Perelli
in a series of foundational papers (see \cite{KaPe99,Kape02,Kape11}
among many others). Our aim is to prove many
of their most important results for our classes. The class $\lsel$
contains all the standard $L$-functions of automorphic representations of 
$GL_n$ over number fields as well as the symmetric square, exterior square and
tensor product $L$-functions, among others. 
Only some of these are even expected to belong to 
$\esel$, and very few of those have actually been proven to do so. 
Finally, the class $\lsel$ contains a number of series that are not part of
the class $\bselp$ defined by A. Booker in \cite{Booker2015} and also contains
Dirichlet series known to belong to $\bselp$.

In Theorems \ref{lselzero}, \ref{lselzerooneempty} and \ref{mainthm} of this paper we 
classify series in $\lsel$ and $\gsel$ of small degrees 
(the notion of degree will be defined below, just after we define the classes)
and present some other results generalising the work of Kaczorowski and
Perelli and Booker. We will apply our results to establish the primitivity of $L$-functions
attached to Maass cusp forms in Theorem \ref{gltwoprimitive} and to compare the zero sets of pairs 
of $L$-functions associated to (the tensor products of) representations
of $\glnak$ in Theorem \ref{zerothm}. The latter result improves on the results of Booker.

Our first task is to define our new classes of Dirichlet series below,
deferring a more detailed examination of our motivations to the next section.

Let $F(s)$ be a non-zero meromorphic function on $\C$. 
We consider the following conditions on $F(s)$.
\begin{enumerate}
\item [(P1)] The function $F(s)$ is given by a
Dirichlet series $\sum_{n=1}^{\infty}\frac{a_n}{n^{s}}$ 
with abscissa of absolute convergence $\nu\ge 1/2$.
\item[(P2)] There is a polynomial $P(s)$ such that the function $P(s)F(s)$ extends to
an entire function. 
\item[(P3)] There exist a real number $Q>0$, a complex number $\omega$ such that 
$\vert\omega\vert=1$, and a function $G(s)$ of the form
\begin{equation}\label{gammafactor}
G(s)=\prod_{j=1}^{r}\Gamma(\lambda_j s+\mu_j)
\prod_{j^{\prime}=1}^{r^{\prime}}\Gamma(\lambda_{j^{\prime}}^{\prime} s+\mu_{j^{\prime}}^{\prime})^{-1},
\end{equation}
where $\lambda_j,\lambda_{j^{\prime}}^{\prime}>0$ and $\mu_j,\mu_{j^{\prime}}^{\prime}\in \C$, such that
\begin{equation}\label{fnaleqn}
\Phi(s):=Q^{s}G(s)F(s)=\omega\overline{\Phi(1-\bar{s})}.
\end{equation}
\item[(P4)] The function $F(s)$ can be expressed as a product
$F(s)=\prod_p F_p(s)$, where 
\begin{equation}\label{prod}
\log F_p(s)=\sum_{k=1}^{\infty} \frac{b_{p^k}}{p^{ks}}
\end{equation}
with $\vert b_{p^k}\vert \le Cp^{k\theta}$ for some $\theta>0$ and some
constant $C>0$.
\end{enumerate}

We will denote by $\lsel$ the class of Dirichlet series  satisfying (P1)-(P3). 
The class of series satisfying (P1)-(P4) will be denoted
$\lselp$. We note that the multiplication of Dirichlet series gives $\lsel$ the structure of a 
monoid and $\lselp$ the structure of a submonoid.
The class of series $\lsel(\nu_0)$ is defined as the those series in $\lsel$ with abscissa of
absolute convergence $\nu$ for some $\nu\le \nu_0$. We see that $\lsel(\nu_0)$ also forms a submonoid
of $\lsel$. 

Given $F(s)$ satisfying an equation of the form \eqref{fnaleqn}, we define
the degree of $F(s)$ to be
$d_F=2(\sum_{j=1}^{r}\lambda_j-\sum_{j^{\prime}=1}^{r^{\prime}}\lambda_{j^{\prime}}^{\prime})$.
We see (in Theorem \ref{uniquefnaleqn}) that this notion is well defined, that is,
it does not change if we take some other functional equation of the form \eqref{fnaleqn}
satisfied by $F(s)$. This justifies the notation $d_F$.

\begin{remark} We can actually allow
an even more general functional equation than \eqref{fnaleqn}. 
With all the other notation being the same as in (P3), we define
$\Psi(s)=G(s)B(s)$, where $B(s)$ is a Dirichlet series convergent in some
right half-plane. If we assume 
a weaker functional equation $\Phi(s)=Q^s\Psi(1-s)$ (the constant $\omega$ can
be absorbed into $B(s)$), it  can be checked that the theorems and the proofs in this paper go 
through in this generality with minor modifications. We avoid doing this for notational simplicity. 
\end{remark}
We may strengthen the condition (P3) to a stronger condition (P3') by further requiring
\begin{equation}\label{gselbound}
\re\left(-\frac{\mu_j}{\lambda_j}\right), 
\re\left(-\frac{\mu_{j^{\prime}}^{\prime}}{\lambda_{j^{\prime}}^{\prime}}\right)<\frac{1}{2},
\,\,1\le j\le r,\,\,1\le j^{\prime}\le r^{\prime}.
\end{equation}
We denote the class of series in $\lsel(1)$ satisfying (P3') 
by $\gsel$. When $\theta<1/2$ in (P4), we obtain the
new condition (P4'), and the class of series satisfying this additional hypothesis
will be denoted $\gselp$.
Occasionally, we may work with the hybrid class 
$\lselp^{\infty}$ satisfying (P1), (P2) and (P3') or the hybrid class 
$\lselp^{\rm{f}}$ satisfying (P1), (P2), (P3) and (P4').

We recall that if we assume $F(s)$ satisfies (P1) with abscissa of convergence 
$\nu\le 1$, (P2) for $P(s)=(s-1)^m$ for some $m\ge 0$, 
and (P3) with $r^{\prime}=0$, and the even stronger requirement that
\begin{equation}\label{selbound}
{\re}(\mu_j)\ge 0,\quad 1\le j\le r,
\end{equation}
we obtain the extended Selberg class $\esel$. We denote
this stronger version of (P3) by (P3''). This class in turn contains the 
original Selberg class $\sel$ introduced by Selberg which consists of all
series in $\esel$ satisfying (P4'). 

We recall that Booker has introduced a class $\bselp$ of Dirichlet series 
\cite{Booker2015} similar to the Selberg class. In the next section we 
explain how our classes contrast with $\esel$ and $\bselp$.

\subsection*{Acknowledgements:} The author would like to thank D. Surya Ramana and R. Balasubramanian 
for helpful discussions related to this paper.

\section{Motivation for introducing $\lsel$ and $\gsel$}\label{motivationsec}

From the definitions we have made, it is obvious that
\[
\esel\subset\gsel\subset\lselp^{\infty}\subset\lsel\quad\text{and}\quad \sel\subset\gselp\subset \lselp^{\rm{f}}\subset\lselp
\]
The passage from the extended Selberg class to the class $\lsel$ involves 
relaxing four hypotheses. We will analyse these hypotheses in turn. In what follows
we will use the notation $K$ for a number field and $\ak$ for its
ring of ad\`eles over $K$. We let $\pi$ and $\pi_i$, $i=1,2$ denote unitary cuspidal 
automorphic representations of $\glnak$ and $\glniak$ respectively, while $\sigma$ will denote an automorphic 
(not necessarily cuspidal or unitary) representation of $\glnak$.

\subsection{Allowing an arbitrary abscissa of absolute convergence}
The class $\lsel$ admits series with an arbitrary abscissa of 
convergence $\nu_0$. This yields a much larger class of $L$-functions than $\esel$.
Let $\sigma$ be as above and denote by $L(s,\sigma)$ the standard 
$L$-function associated to $\sigma$. It is known 
by \cite{GoJa72} that these $L$-functions lie in $\lsel$ (and, in fact in 
$\lselp$). However, even 
in some of the simplest cases, and even when the representation $\sigma$ is 
unitary, $L(s,\sigma)$ does not always belong to $\esel$. For instance, the 
$L$-function $\zeta(s-1/2)\zeta(s+1/2)$, which is attached to the trivial 
representation of $\gltwo(\ak)$, does not lie in $\esel$. This 
representation is not {\it globally generic}. Indeed non-generic automorphic
representations provide a large class of $L$-functions that belong
to $\lsel$ but not to $\esel$. The $L$-functions attached to Siegel modular forms
provide further examples of such series.

The difference between the submonoid $\lsel(1)$ and $\esel$ is less obvious,
and indeed, it is conceivable that these classes coincide. Nonetheless, there
are many examples of series that are known to lie in $\lsel(1)$ but have not been
proven to belong to $\esel$ as we discuss below. These examples include the symmetric and exterior square 
$L$-functions $L(s,\vee^2(\pi))$ and $L(s,\wedge^2(\pi))$, and the 
Rankin-Selberg $L$-function $L(s,\pi_1\times\pi_2)$.  One would expect all $L$-functions
associated to generic representations of $\glnak$ to lie in $\lsel(1)$. This is the
Generalised Ramanujan Conjecture which we discuss in Subsection 
\ref{archbounds}.

\subsection{Allowing a finite number of poles} Our condition (P2) allows
the polynomial $P(s)$ to be arbitrary, in contrast to the requirement that
$P(s)=(s-1)^m$ for some integer $m\ge 0$ for functions in $\esel$. This
allows our series to have a finite number of poles, a priori of arbitrary 
orders and locations. Indeed, general automorphic $L$-functions can and do 
have a finite number of poles - the example of $\zeta(s-1/2)\zeta(s+1/2)$
is once again instructive. Under suitable normalisations, one can expect
that the $L$-functions associated to globally 
generic unitary automorphic representations have poles only at $s=1$, but 
there remain cases (using the Langlands-Shahidi method, see \cite{Shahidi88}) where the 
finiteness of poles has been established without the stronger expectation 
having been proved. Even when it has been achieved, the passage from the 
finiteness of poles to holomorphy when $s\ne 1$ has not come easily .

\subsection{Bounds at the archimedean factors.}\label{archbounds}\label{archbounds}
The condition (P3") is the assertion that all the 
trivial zeros of $F(s)$ lie in the left half-plane. {\it Conjecturally}, 
the standard $L$-functions $L(s,\pi)$ satisfy this condition, as do the $L$-functions 
$L(s,\vee^2(\pi))$, $L(s,\wedge^2(\pi))$ and $L(s,\pi_1\times\pi_2)$. But proving this is equivalent to proving the 
Generalised Ramanujan Conjectures at infinity, 
which even in the simplest case of representations associated to Maass cusp forms is the
deep and unproven Selberg Eigenvalue Conjecture. Thus, the class
$\esel$ currently excludes many of the $L$-functions of the greatest interest, given
our current state of knowledge. 

It is possible that an element in $\lsel$ satisfies more than one functional equation
and (P3') would appear to be the weakest assumption we can impose to ensure 
that the factor $G(s)$ is unique (upto multiplication by a scalar).
This motivates, in part, our consideration of the class $\gsel$. 

The class $\gsel$ (and, in fact, the class $\gselp$) 
contains the standard $L$-function $L(s,\sigma)$ when $\sigma$ is a
unitary automorphic representation, since these are known to satisfy (P3') by a theorem of Jacquet 
and Shalika (see \cite{JaSh811} and \cite{JaSh812}). These include the the $L$-functions
of Maass cuspidal eigenforms, for instance.
However, even when both representations are globally generic, the functions $L(s,\pi_1\times\pi_2)$ 
are only conjecturally in $\gselp$. 
The class $\gsel$ is perhaps the largest class of series for which one may expect 
the Generalised Ramanujan Conjectures to hold. 

\subsection{Allowing $r^{\prime}>0$}\label{rprime}
For $L$-functiions in $\esel$, and for the $L$-functions of automorphic 
representations, we always have $r^{\prime}=0$. The motivation for allowing $r^{\prime}\ge 1$,
arises in the following setting. Let $F_j(s)$, $j=1,2$, $F_1(s)\ne F_2(s)$ be elements of $\lselp(1)$. 
Consider the quotient 
$F(s)=F_1(s)/F_2(s)$, and suppose that $d_{F_1}-d_{F_2}=1$. If we assume that this quotient
has at most finitely many poles, then under certain further assumptions we can 
use Theorem \ref{mainthm} to get a contradiction. 
Hence, in a number of situations we will be able to prove that the zero sets of 
distinct $L$-functions are very different - that there are infinitely many zeros (counted with multiplicity) 
of $F_2(s)$ which are not zeros of $F_1(s)$.
This idea was first pursued in special cases (when $r^{\prime}=0$) 
in \cite{Ragh99} and has been generalised by Booker to his class $\bselp$
in \cite{Booker2015}. Theorem \ref{zerothm} of this paper further
generalises Booker's theorem to our class $\lselp(1)$. Thus, the motivation for 
weakening this hypothesis comes partly from having a specific application in mind.
Conjecturally, of course,
we believe that any element of $\lsel$ will satisfy a functional equation 
in which $r^{\prime}=0$. 

\subsection{The Euler product axioms} The weakening of (P4') to (P4) allows the presence of non-trivial
degree $0$ elements which is often a minor irritant, especially when one is
interested in the factorisation of series. For this reason, we prefer to
retain the condition $\theta<1/2$ and work with the class $\gselp$ in this paper 
when discussing the primitivity of elements in these monoids. We may also work
in $\lselp^{\rm{f}}$, when the condition at infinity is not known
to hold or is not required for the proofs.

The condition (P3') is exactly the archimedean analogue of the condition $\theta<1/2$ 
in (P4') and is known to hold $L(s,\pi)$ by \cite{JaSh811,JaSh812}.
Thus, in the class $\gselp$, the bounds for local 
parameters at both the archimedean and nonarchimedean places are assumed to be 
less than $1/2$.
From a purely aesthetic point of view,
the imposition of the same Jacquet-Shalika type bound at both archimedean and non-archimedean places 
seems to make $\gselp$ a natural subclass in the theory. Further, 
$\gselp\subset\lselp^f\cap\lselp^{\infty}$ so theorems proved for these {\it hybrid classes}
remain valid for $\gselp$.

Finally, we must add that analogues of (various versions of) the Generalised Ramanujan Conjectures would assert that
the classes $\sel$, $\gselp$ and $\lselp^{\rm{f}}$, coincide, while the most general {\it Converse Theorem}
would assert that $\lselp$ coincides with the class of standard automorphic $L$-functions of $\glnaq$, $n\ge 1$. In addition, one 
could also conjecture $\esel=\gsel=\lsel(1)$. Until these are proven however, it is obviously best to work in the largest
possible class.

\subsection{A comparison with the class $\bselp$} 
In \cite{Booker2015}, Booker introduces what he calls an $L$-datum using the 
framework of explicit formulas with which he associates a Dirichlet
series. This class of Dirichlet series forms a monoid under multiplication. 
We will not describe his hypotheses in detail but make the following observations. 
A requirement for a series to belong to the class $\bselp$ is that the mean square of the coefficients 
$a_n$ satisfies the Ramanujan bound on average. This is stronger than the average 
bound ($\nu=1$) that we have assumed for the class $\gsel$.
For the standard $L$-function $L(s,\pi)$ this stronger bound follows
from the papers of Jacquet and Shalika referenced above. 
However, as in Subsection \ref{archbounds}, the
$L$-functions $L(s,\vee^2(\pi))$ and $L(s,\wedge^2(\pi))$ and $L(s,\pi_1\times\pi_2)$ have 
not been proven to lie 
in $\bselp$. 

Although Booker does not assume the existence of an Euler product explicitly, 
his series are defined as exponentials of other series, and thus satisfy a 
non-vanishing condition, implicitly
invoking a condition very close to that of the existence of an Euler product. Thus our 
class $\gsel$ is not contained in 
$\bselp$, since we may take linear combinations of elements in $\gselp$ with the 
same gamma factors to produce elements in $\gsel$ which are not in $\bselp$.
In light of the Generalised Ramanujan Conjecture at infinity, one 
might expect the class $\bselp$ to coincide with our class $\gselp$. However, 
the uniqueness of the functional equation for elements in 
$\bselp$ does not seem to follow immediately.

{\it A priori}, the kernels in an $L$-datum allow for more general gamma factors 
at infinity than the quotients we allow for the class $\lsel$. It seems likely that
any Dirichlet series in $\bselp$ will satisfy a functional equation with the simpler 
quotients of gamma factors that arise in $\lsel$, probably with $r^{\prime}=0$.
Perhaps establishing a suitable local
functional equation will allow the passage from the more general condition
to the more restrictive one.

To summarise, one expects that $\bselp=\gselp$. There are no known elements
of $\bselp$ which do not lie in $\gselp$. However, there are many elements of
$\lselp(1)$ which have not been proven to belong to $\bselp$.
Finally, not imposing the condition $\nu_0=1$ allows us to include the naturally occurring 
non-generic automoprhic $L$-functions in the class $\lselp$, and these do not 
belong to $\bselp$. Linear combinations of series in $\gselp$ and $\lselp$ will
produce series in $\gsel$ and $\lsel$ which will also not belong to $\bselp$.
\subsection{Nomenclature} The notation $\lsel$ is supposed to be suggestive 
of the the word ``automorphic", 
while $\gsel$ suggests ``generic" for the generic automorphic representations that are expected 
to be the source of all functions in this class. The Ramanujan conjecture is 
roughly the statement that globally generic representations should
be tempered. The condition $\nu_0=1$ is consistent with this expectation.

\section{Preliminaries}\label{prelim}
In the rest of this paper, for any function $a:\C\to \C$, we will use the notation $\tilde{a}(s):=\overline{a(\bar{s})}$.

We will require two different variants of Stirling's formula applied to the quotients of the gamma functions
appearing in the functional equation. Let $w=u+iv$ be in $\C$. The simplest avatar of the 
formula that we will use is 
\begin{equation}\label{zerothgammaest}
\frac{\tilde{G}(1/2-it-w)}{G(1/2+it+w)}\ll (1+|t+v|+|u|)^{-u},
\end{equation}
where $G(z)$ is the function appearing in \eqref{gammafactor}.
A more refined version of Stirling's formula gives 
\begin{flalign}\label{thirdgammaest}
\frac{\Gamma\left(\lambda_k(1/2-it) +\bar{\mu}_k\right)}{\Gamma\left(\lambda_k(1/2+it) +\mu_k\right)}
\sim & e^{-2\lambda_kit\log\frac{t}{2e}-2\lambda_kit(\log\lambda_k+1)+(\bar{\mu}_k-\mu_k)\log t}
\nonumber\\
&\times e^{(\bar{\mu}_k-\mu_k)\log\lambda_k-(\mu_k+\bar{\mu}_k+\lambda_k)i\frac{\pi}{2}+(\mu_k-\bar{\mu}_k)}
\cdot\left[1+O(1/t)\right],
\end{flalign}
for $1\le k=j\le r$. The same formula holds for $1\le k=j^{\prime}\le r^{\prime}$ when 
$\lambda_k$ and $\mu_k$ are replaced by $\lambda_{k^{\prime}}^{\prime}$ and 
$\mu_{k^{\prime}}^{\prime}$ respectively. Taking the product 
over all $j,j^{\prime}$ in \eqref{gammafactor} yields
\begin{equation}\label{fourthgammaest}
\frac{\tilde{G}(1/2-it)}{G(1/2+it)}
\sim e^{-it\log\frac{t}{2e}}t^{iA}e^{iB}C^{-it}\cdot 
\left[1+O(1/t)\right],
\end{equation}
where 
\[
A=-i[(\bar{\mu}-\mu)-(\overline{\mu^{\prime}}-\mu^{\prime})],C=e\prod_{j,j^{\prime}=1}^{r,r^{\prime}}{\lambda_j}^{2\lambda_j}
{\lambda_{j^{\prime}}^{\prime}}^{-2\lambda_{j^{\prime}}^{\prime}}
\]
and
\begin{flalign}
B=&-i\left[\sum_{j=1}^r(\bar{\mu}_j-\mu_j)\log\lambda_j-\sum_{j^{\prime}=1}^r(\overline{\mu_{j^{\prime}}^{\prime}}-\mu_{j^{\prime}}^{\prime})\log\lambda_{j,}\right]\nonumber\\
&-(\mu-\bar{\mu})+(\mu^{\prime}-\overline{\mu^{\prime}})-((\mu-\bar{\mu})-(\mu^{\prime}-\overline{\mu^{\prime}})+1)\frac{\pi}{2},
\end{flalign}
where 
\[
\mu=\sum_{j=1}^{r}\mu_j\quad\text{and}\quad \mu^{\prime}=\sum_{j^{\prime}=1}^{r^{\prime}}\mu_{j^{\prime}}^{\prime}.
\]
Note that $A\in \R$ and $C>0$.

In addition, we will also require the following lemma which expresses the function
$F(s)$ as a convergent Dirichlet series at $s=1/2+it$ plus two error terms over which
one has relatively good control.
\begin{lemma}\label{basiclemma} Let $w=u+iv$ and $0<\eta<1$, and assume
that $F(s)\in \lsel_1(\nu_0)$.
We have
\begin{equation}\label{basiclemmaeqn}
F(1/2+it)=\sum_{n=1}^{\infty}\frac{a_ne^{-n/X}}{n^{1/2+it}}
+r_1+r_2
\end{equation}
where $r_1:=r_1(t,X)=O(X^{\nu_0-1/2}e^{-|t|})$ is
identically zero if $F(s)$ is entire and
\begin{equation}\label{rtwoz}
r_2=\frac{1}{2\pi i}\int_{u=-1+\eta}F(1/2+it+w)X^{w}\Gamma(w) dw\ll
O((1+|t|)^{1-\eta}X^{-1+\eta}),
\end{equation}
and where $u=-1+\eta$ is a line on which none of the poles of $F(s)$ lie. 
\end{lemma}
\begin{proof}
When $c>\nu_0$, we have 
\[
\frac{1}{2\pi i}\int_{u=c}F(1/2+it+w)X^w\Gamma(w)dw
=\sum_{n=1}^{\infty}\frac{a_ne^{-n/X}}{n^{1/2+it}}.
\]
When we move the line of 
integration from $u=c$ to $u=-1+\eta$ we will cross the poles of the integrand of the form 
$w=\beta-1/2-it$, where $\beta$ is a pole of $F(s)$, and also the pole at $w=0$. 
The residue at $w=0$ is $F(1/2+it)$, and the residue at $\beta-1/2-it$ is 
majorised upto a constant by the $X^{\beta-1/2}e^{-|t|}$, since 
$\Gamma(\beta-1/2-it)=O(e^{-|t|})$. For any pole $\beta$, 
we must have $\re(\beta)\le \nu_0$, since $\nu_0$ is the abscissa of 
absolute convergence. Hence, the residue at $\beta-1/2-it$ will also be majorised by 
$X^{\nu_0-1/2}e^{-|t|}$. We denote the sum of the residues by $r_1$. It is 
obviously identically zero if $F(s)$ is entire. 

From the functional equation we have
\[
F(1/2+it+w)=\omega Q^{-2it}\frac{\tilde{G}(1/2-it)}{G(1/2+it)}\tilde{F}(1/2-it-w).
\]
Substituting this in $r_2$, the estimate \eqref{rtwoz}
follows immediately from \eqref{zerothgammaest} . 
\end{proof}

\section{Classifying series of small degree: $0\le d <1$}\label{generalities}

We first show that the notion of degree is well-defined for functions in the class $\lsel$.
Indeed, suppose
that $F(s)$ in $\lsel$ satisfies two different functional equations
\[
\Phi_j(s):=Q_j^{s}G_j(s)F(s)=\omega_j\tilde{\Phi_j}(1-s),\quad j=1,2.
\]
Let $d_j$, $j=1,2$, be the degrees of $G_j(s)$. Taking the quotient $\Phi_1(s)/\Phi_2(s)$, we see that
\[
H(s)=\frac{G_1(s)}{G_2(s)}=\frac{\omega_1\tilde{G}_1(1-s)}{\omega_2\tilde{G}_2(1-s)}=c\cdot\tilde{H}(1-s).
\]
The left-hand side is holomorphic and non-vanishing for $\res\gg 0$, 
while the right-hand side is holomorphic and non-vanishing for $\res\ll 0$. 
Further, all the zeros and poles (on both sides) lie on a finite number of rays
parallel to the real (horizontal) axis. Thus, the zeros and poles of $H(s)$ must lie on a finite
number of line segments contained in a bounded vertical strip, whence it 
follows that $H(s)$ has at most finitely many zeros and poles. 
The number of poles of $G_j(s)$ with $\re(s)>-T$ is asymptotic to $d_jT$ for 
$j=1,2$. Hence, if 
$d_1\ne d_2$, $H(s)$ must have infinitely many zeros or poles. 
It follows that we must have $d_1=d_2$. Hence, we see that the degree of an element in $\lsel$ 
is well-defined. We are thus justified in using the notation $d_F$ to denote the degree of $F$.

If $F(s)$ is in $\lselp^{\infty}$, we follow the arguments of \cite{CoGh93}. We see that $H(s)$ is an entire function without zeros, 
since the poles and zeros of $G_1(s)/G_2(s)$ lie in the half-plane $\res<1/2$, 
while those of $\tilde{G_1}(1-s)/\tilde{G}_2(1-s)$ lie in $\res>1/2$. Since $H(s)$ is of
order $1$ (being the quotient of functions of order $1$), it must have the form $e^{as+b}$ for constants $a$ and $b$.
Because of the functional equation, we see that the constant $a$ is purely imaginary. Stirling's 
formula shows that the quotient $|H(it)/H(-it)|\to 1$ as $t\to\infty$, which shows that
$a=0$. We can summarise our arguments as
\begin{theorem}\label{uniquefnaleqn}
For $F(s)$ in $\lsel$, the degree $d_F$ is well defined. If $F(s)$ in 
$\lselp^{\infty}$ satisfies two different functional equations
\[
\Phi_j(s):=Q_j^{s}G_j(s)F(s)=\omega_j\overline{\Phi_j(1-\bar{s})},\quad j=1,2,
\]
there is a constant $c$ such that $Q_1^{s}G_1(s)=cQ_2^{s}G_2(s)$.
\end{theorem}
In view of the theorem above, we can define $\lseld$ to be the subset consisting 
of series $F(s)$ in $\lsel$ with $d_F=d$.

We begin by classifying the elements in $\lseld$ when $0\le d\le 1$. To this end
we first prove the following proposition following the proof for $\sel$ in
\cite{CoGh93}. 
\begin{proposition}\label{absconv} If $0\le d_F<1$, the series $F(s)$ is absolutely convergent
on the whole complex plane and hence gives rise to an entire function.
\end{proposition}
\begin{proof}
For $\re(s)=c>\nu_0$, we know that
\[
h(y)=\sum_{n=1}^{\infty}a_ne^{-2\pi ny}=\int_{\re(s)=c}F(s)\Gamma(s)(2\pi y)^{-s}ds.
\]
Shifting the line of integration to the left and letting $c\to-\infty$, we get
\[
\sum_{n=1}^{\infty}a_ne^{-2\pi ny}=
\sum_{n=0}^{\infty}\frac{(-1)^nF(-n)(2\pi y)^n}{n!}
+\sum_{k=1}^{l}P_k(\log y)y^{\beta_k}.
\]
where $-\beta_k$ are the poles of $F(s)$ and the $P_k$ are polynomials
(with degree one less than the order of the pole at $\beta_k$) for $1\le k\le l$.
Using the functional equation \eqref{fnaleqn}, we obtain
\[
\sum_{n=1}^{\infty}a_ne^{-2\pi ny}=
\sum_{n=0}^{\infty}\frac{(-1)^n G(n+1)F(n+1)(2\pi y)^n}{G(-n)n!}+
\sum_{k=1}^{l}P_k(\log y)y^{\beta_k}.
\]
Using \eqref{zerothgammaest} one checks easily that $G(n+1)/G(-n)n!=O(n^{-(1-d)n}K^n)$ 
for some constant $K$, so the infinite sum on the right above converges. 
It follows that $h(y)$ extends to a holomorphic function on $\C\setminus (-\infty,0]$. 
But we also know that $h(y)$ is periodic with period $i$, so it extends holomorphically
to the real axis as well, and thus to an entire function $h(z)$.
Thus, we can write 
\[
a_ne^{-2\pi ny}=\int_{0}^{1}h(ix-y)e^{-2\pi inx}dx.
\]
Differentiating repeatedly with respect to $y$ shows that
$a_n\ll n^{-k}$, for any $k\ge 0$. It follows that the Dirichlet
series $\sum_{n=1}^{\infty}a_nn^{-s}$ converges absolutely (and uniformly
on compact subsets) everywhere in the complex plane. It is thus an entire 
function.
\end{proof}
\begin{theorem}\label{lselzero} The class $\lsel_0$ consists of 
Dirichlet polynomials of the form 
\[
F(s)=\sum_{n\,|\,Q_1}\frac{a_n}{n^s},
\]
for some $Q_1>0$ an integer. In fact, $\lsel_0=\esel_0$.
We also have $\gselp_0=\{1\}$.
\end{theorem}
\begin{proof} 
We first assume that $r^{\prime}\ne 0$ for
every non-constant factor $G(s)$ that arises in a functional equation of the form
\eqref{fnaleqn} for $F(s)$. We write $G(s)=g_1(s)/g_2(s)$ where
\[
g_1(s)=\prod_{j=1}^{r}\Gamma(\lambda_j s+\mu_j)\,\,\text{and}\,\, 
g_2(s)=\prod_{j^{\prime}=1}^{r^{\prime}}\Gamma(\lambda_{j^{\prime}}^{\prime} s+\mu_{j^{\prime}}^{\prime}).
\]
Our first task is to show that $G(s)$ has only finitely many zeros and poles.
Let $\alpha>0$ be a real number and $\beta\in\C$. 
We will refer to a subset of $\C$ of the form 
$W=\{-(n+\beta)/\alpha \,\vert \,n\in \N\cup\{0\}\}$ as a $\gamma$-set,
or more precisely, as a $\gamma(\alpha,\beta)$-set. We see that
the multiset of zeros of $G(s)$ is the finite sum of $\gamma$-sets.
Notice that the elements of a $\gamma$-set lie on a horizontal 
ray parallel to the negative real axis.

Suppose that $G(s)$ has infinitely many zeros. These will lie on a finite number of
rays parallel to the negative real axis. Since $F(s)$ has at most finitely many
poles, all but finitely many of the zeros of $G(s)$ will be poles of $\Phi(s)$,
and hence, of $\tilde{\Phi}(1-s)$. The zeros of $\tilde{G}(1-s)$ lie on finitely
many rays parallel to the positive real axis. Hence, only finitely many of
these will lie in any left half-plane. It follows that $\tilde{F}(1-s)$
has infinitely many zeros lying on some ray parallel to the negative real axis.
This is impossible, since $\tilde{F}(1-s)$, being absolutely convergent
in some left half-plane, is dominated by its first non-zero term and 
has no zeros for $\re(s)<-c_0$ for some $c_0>0$. Thus $G(s)$ has only finitely many zeros.

We now show that $G(s)$ has only finitely many poles, that is, that all but 
finitely many of the poles of $g_1(s)$ are poles of $g_2(s)$. 
Let $H_T$ denote the half-plane $\res>-T$, and
let $N_S(T)=|S\cap H_T|$ for any discrete subset $S$ of $\C$. We denote 
the $\lim_{T\to\infty}N_S(T)/T$ by $D(S)$, if it exists, and call it the 
density of $S$.
Note that $D(W)=\alpha$ for the $\gamma(\alpha,\beta)$-set $W$. 
The notion of density generalises naturally to multisets, in particular to sums of $\gamma$-sets.
If $V=\sum_{i=1}^rW_i$ is a multiset which is the sum of the sets $W_i$ each with
density $\alpha_i$, then $D(V)=\sum_{i=1}^rD(W_i)$.
We will need the following lemma.
\begin{lemma} Let $W_i=\{(-n-\beta_i)/\alpha_i \,\vert\,n\in \N\cup\{0\}\}$, $i=1,2$ be
$\gamma$-sets. Then either 
$|W_1\setminus W_2|<\infty$ or $D(W_1\setminus W_2)>0$. Further,
if $D(W_1\setminus W_2)>0$, $W_1\setminus W_2$ is a union
of $\gamma$-sets and a finite set.
\end{lemma}
\begin{proof} 
Either $|W_1\cap W_2|\le 1$ or 
$|W_1\cap W_2|\ge 2$. In the first case there is nothing to prove.
In the second case, it is easy
to see that $\alpha_1=\alpha_2m_1/m_2$
and $\beta_1+k=\beta_2m_1/m_2$
for coprime integers $m_1$ and $m_2$ and some integer $k$.
The elements of $W_1\cap W_2$ have (except possibly for a finite set) the form
$-[m_1n+a+\beta_1]/\alpha_1$ for some integer $a$ with 
$0\le a<m_1-1$. If $m_1=1$, we see that $|W_1\setminus W_2|<\infty$.
If $m_1-1>1$, $W_1\setminus W_2$ is (upto a finite set) a union
of sets of the form $-[m_1n+b+\beta_1]/\alpha_1$, 
$n\in \{0\}\cup \N$ and thus a union of $\gamma$-sets and
a finite set. Further $D(W_1\setminus W_2)=(m_1-1)D(W_1)/m_1=(m_1-1)\alpha/m_1>0$.
This proves the lemma.
\end{proof}
\begin{corollary}
If $V_1=\sum_{i=1}^rW_i^1$ and $V_2=\sum_{j=1}^sW_j^2$ are the sums of $\gamma$-sets viewed
as multisets (that is, if an element occurs in $k$ different sets $W_i^1$ or $W_j^2$, it is thought of
as occurring $k$ times in the sum $V_1$ or $V_2$) , then
either $|V_1\setminus V_2|<\infty$ or $D(V_1\setminus V_2)>0$. In the latter
case $V_1\setminus V_2$ is a union of $\gamma$-sets and a finite set.
\end{corollary}
\begin{proof} Indeed, we simply apply the lemma succesively to the differences
$W_i^1\setminus W_j^2$ as $1\le i\le r$ and $1\le j\le s$. The subtractions are performed
in lexicographic order on the pairs $(i,j)$. After each subtraction,
the lemma above shows that what remains of each of $W_i^1$ and $W_j^2$ is a finite set or 
a finite union of $\gamma$-sets and (possibly) a 
finite set, and we can thus continue the process with the next subtraction.
\end{proof}

We apply our corollary to the following situation.
The multiset $V_1$ of poles
of $g_1(s)$ is the sum of the $\gamma$-sets $S_j=\{(-n-\mu_j)/\lambda_j\}$, $1\le j\le r$, while the
multiset of poles $V_2$
of $g_2(s)$ is the sum of the $\gamma$-sets $S^{\prime}_{j^{\prime}}=
\{(-n^{\prime}-\mu^{\prime}_{j^\prime})/\lambda^{\prime}_{j^{\prime}}$, 
$1\le j^{\prime}\le r^{\prime}$,
where $n$ and $n^{\prime}$ run through the non-negative integers.
We have already shown that $G(s)$ has only finitely many zeros.
Hence, $|V_2\setminus V_1|<\infty$ so $D(V_1)\ge D(V_2)$.
If $V_1\setminus V_2$ is not finite, the corollary 
above tells us that $D(V_1\setminus V_2)>0$. Thus we find
that $D(V_1)>D(V_2)$. But by assumption, the degree of $G(s)=D(V_1)-D(V_2)=0$.
This gives a contradiction. Thus $G(s)$ has at most finitely many poles,
which is what we have been trying to prove.

Since $G(s)$ has at most finitely many zeros and poles, and is a quotient of
functions of order $1$, it can be replaced in the functional
equation by a factor of the form $AB^sR(s)$ for some rational function $R(s)=p(s)/q(s)$, where
$p(s)$ and $q(s)$ are monic polynomials, $A\in \C$ and $B>0$. Cross multiplying, 
we get an equation of the form
\[
\tilde{p}(1-s)q(s)F(s)=\omega Q_1^{s}p(s)\tilde{q}(1-s)\tilde{F}(1-s),
\]
for $Q_1>0$. Using Perron's formula (as in  \cite{KaPe99}) it follows immediately 
that $Q_1$ is an integer and that $F(s)$ is a Dirichlet polynomial with nonzero 
coefficients only when $n\mid Q_1$. It also easily follows as a 
consequence that $\tilde{p}(1-s)q(s)=p(s)\tilde{q}(1-s)$. Thus $F(s)$
actually satisfies a functional equation of the form
\[
F(s)=\omega Q_1^{s}\tilde{F}(1-s),
\]
so $F(s)$ lies in $\esel$.

If $F(s)\in \gselp_0$, we know further that $F_p(s)$ is non-vanishing for $\res\ge 1/2$, 
and hence, that $F(s)$ which is a product of at most finitely many $F_p(s)$, is also
non-vanishing in this half-plane. By the functional equation it is non-vanishing
in $\res\le 1/2$ as well. It follows that $F(s)$ is entire and non-vanishing, and 
since $a_1=1$, we must have $F(s)\equiv 1$.
\end{proof}
\begin{remark} We have proved the analogue of first part of Theorem 1 of 
\cite{KaPe99}. The second part of Theorem 1 of that paper gives a somewhat
more precise description of the elements of $\esel_0$ in terms of invariants
$q$ and $\omega^{*}$ that Kaczorowski and Perelli associate to elements of $\esel_0$.
We can thus recover a similar sharper statement for elements of $\lsel_0$.
\end{remark}
\begin{theorem}\label{lselzerooneempty} 
 If $0<d<1$, then $\lseld=\emptyset$.
\end{theorem}
For the Selberg class this is a theorem of Richert \cite{Richert57} and Conrey 
and Ghosh (\cite{CoGh93}) and we follow the proof in the latter paper.
\begin{proof}
By Proposition \ref{absconv} we know that $F(s)$
is uniformly bounded in $\re(s)>-\nu$ since it is absolutely convergent in every half-plane.
The functional equation 
\[
F(s)=\frac{\tilde{G}(1-s)}{G(s)}Q^{1-2s}\tilde{F}(1-s)
\]
and Stirling's formula show that $F(s)$ cannot be bounded on the vertical line $\res=-\varepsilon$
for any $\varepsilon>0$.
\end{proof}

\section{The case $d_F=1$}
The main result of this section is
\begin{theorem}\label{mainthm} 
Suppose that $F(s)$ is in $\lsel_1(1)$. 
\begin{enumerate}
\item There exists $A\in \R$ 
and an integer $q>0$ such that $a_nn^{-iA}$ is periodic with period $q$. Further,
\begin{equation}\label{lincomb}
F(s)=\sum_{\chi\Mod{q}}P_{\chi}(s)L(s+iA,\chi^{*}),
\end{equation}
where the sum runs over all Dirichlet characters $\chi\Mod{q}$, $P_{\chi}\in \lsel_0$ and
$\chi^{*}$ is the primitive Dirichlet character inducing $\chi$.
\item
If $F(s)$ is in $\lselp_1(1)$, 
there is a Dirichlet character $\chi\Mod{q}$ such that 
$F_p(s)=(1-\chi_pp^{-s+iA})^{-1})$ for all $p\nmid q$. If
further $F(s)$ is in $\gselp$, $F(s)=L(s+iA,\chi)$ for a primitive Dirichlet character
$\chi\Mod{q}$.
\end{enumerate}
\end{theorem}
Kaczorowski and Perelli \cite{KaPe99} proved the theorem above for series in $\esel$.
Soundararajan gave another proof of their theorem in \cite{Sound05}, but assuming
$a_n\ll n^{\varepsilon}$. We present a modified version of Soundararajan's
proof below for class $\lsel_1(1)$. Note that the theorem above is valid for $\gsel_1\subset \lsel(1)$.
We hope to remove the restriction $\nu_0=1$ in future work. 

The hard part of proving the theorem above lies in proving the first assertion of the first part of the theorem.
The other assertions will follow relatively easily after that.
\begin{proof} 
Since $F(s)$ converges absolutely for
$\res>1$, we have the estimate
\begin{equation}\label{averageram}
\sum_{n<X}\frac{|a_n|}{\sqrt{n}}=O(X^{\frac{1}{2}+\varepsilon}).
\end{equation}
Define
\begin{equation}\label{falphat}
\eff(\alpha,T)=\frac{1}{\sqrt{\alpha}}\int_{\alpha T}^{2\alpha T}F(1/2+it)e^{it\log\frac{t}{2\pi e\alpha}-i\frac{\pi}{4}}dt. 
\end{equation}
The proof involves showing that 
\[
\eff(\alpha):=\lim_{T\to\infty}\frac{\eff(\alpha,T)}{T^{1+iA}}
\]
exists, and that $\eff(\alpha+1)=\eff(\alpha)$, 
that is, $\eff(\alpha)$ is periodic of period $1$.

Using the functional 
equation \eqref{fnaleqn} and equation \eqref{fourthgammaest}, we have
\[
\eff(\alpha,T)=\frac{\omega e^{iB}}{\sqrt{\alpha}}\int_{\alpha T}^{2\alpha T}
\tilde{F}(1/2-it)(C[\pi Q^2\alpha])^{-it}t^{iA}\left[1+O(1/t)\right]dt.
\]
Now we follow \cite{Sound05}, but we use Lemma \ref{basiclemma} for $\nu_0=1$ instead, to get
\begin{flalign}
\eff(\alpha,T)&=\frac{\omega e^{iB}}{\sqrt{\alpha}}\int_{\alpha T}^{2\alpha T}
\sum_{n=1}^{\infty}\frac{\overline{a_n}}{\sqrt{n}}e^{-n/X}(n^{-1}C[\pi Q^2\alpha])^{-it}t^{iA}
\left[1+O(1/t)\right]dt\nonumber\\ 
&+O(X^{\frac{1}{2}}e^{-\alpha T})+O(T^{2-\eta}X^{-1+\eta})\nonumber
\end{flalign}
\begin{flalign}\label{postbasic}
&=\frac{\omega e^{iB}}{\sqrt{\alpha}}\sum_{n=1}^{\infty}\frac{\overline{a_n}}{\sqrt{n}}e^{-n/X}
\int_{\alpha T}^{2\alpha T}(n^{-1}C[\pi Q^2\alpha])^{-it}t^{iA}\left[1+O(1/t)\right]dt\nonumber\\
&+O(X^{\frac{1}{2}}e^{-T})+O(T^{2-\eta}X^{-1+\eta}).
\end{flalign}
If we choose $X=T^{4/3}$, then both the error terms above are
$O(T^{2/3+\varepsilon})$ for any $\varepsilon>0$, if $\eta$ is chosen small enough.
It is easy to see that
\[
\int_{\alpha T}^{2\alpha T}(n^{-1}C[\pi Q^2\alpha])^{-it}t^{iA}\cdot O(1/t)dt=O(1).
\]
Using \eqref{averageram}, we see that the contribution of this integral to the sum in \eqref{postbasic}
is $O(X^{1/2+\varepsilon})$. With $X=T^{4/3}$ as above, we have
$O(X^{1/2+\varepsilon})=O(T^{2/3+\varepsilon})$, for any $\varepsilon>0$.
It remains to estimate 
\[
\frac{\omega e^{iB}}{\sqrt{\alpha}}\sum_{n=1}^{\infty}\frac{\overline{a_n}}{\sqrt{n}}e^{-n/X}
\int_{\alpha T}^{2\alpha T}(n^{-1}C[\pi Q^2\alpha])^{-it}t^{iA}dt.
\]
The integral is estimated using integration by parts (as in (4a) of of \cite{Sound05}). 
Together with the estimate \eqref{averageram},  we see that 
for $n^{-1}C[\pi Q^2\alpha]\ne 1$, the sum above is once again majorised 
by $O(X^{1/2+\varepsilon})=O(T^{2/3+\varepsilon})$ for any $\varepsilon>0$.
Thus, for $n^{-1}C[\pi Q^2\alpha]\ne 1$ we have 
\begin{equation}\label{alphanotint}
\eff(\alpha,T)=O(T^{\frac{2}{3}+\varepsilon})
\end{equation}
for some $\varepsilon>0$.
If $m^{-1}C[\pi Q^2\alpha]=1$, for some integer $m$, we get
\[
\int_{\alpha T}^{2\alpha T}(m^{-1}C[\pi Q^2\alpha])^{-it}t^{iA}dt=
\frac{2^{1+iA}-1}{1+iA}\alpha^{1+iA}T^{1+iA}.
\]
Combining \eqref{postbasic} with the estimates above, we see that for a suitable $\eta$, for any $\varepsilon>0$ 
we have
\begin{equation}\label{alphaint}
\eff(\alpha,T)=\omega e^{iB}
\frac{\overline{a_m}}{(C\pi Q^{2})^{\frac{1}{2}}}\frac{2^{1+iA}-1}{1+iA}
\alpha^{iA}T^{1+iA}+O(T^{\frac{2}{3}+\varepsilon}).
\end{equation}
Dividing by $T^{1+iA}$ in \eqref{alphanotint}
and taking the limit as $T\to\infty$, we see that $\eff(\alpha)=0$. 
Similarly, if we divide \eqref{alphaint} by $T^{1+iA}$ and take the limit as $T\to\infty$,
\[
\eff(\alpha)=\omega e^{iB}
\frac{\overline{a_m}\alpha^{iA}}{(C\pi Q^{2})^{\frac{1}{2}}}
\cdot\frac{2^{1+iA}-1}{1+iA}.
\]
Combining these two cases,
\begin{equation}\label{falphaeqn}
\eff(\alpha)=\omega e^{iB}\delta(C[\pi Q^2\alpha]=m)
\frac{\overline{a_{C[\pi Q^2\alpha]}}\alpha^{iA}}{(C\pi Q^{2})^{\frac{1}{2}}}
\cdot\frac{2^{1+iA}-1}{1+iA}.
\end{equation}
We have thus shown that the desired limit exists. 

We will now
show that $\eff(\alpha)$ is a periodic function. Once again, we use 
Lemma \ref{basiclemma} with $\nu_0=1$ and $X=T^{4/3}$. With the same arguments as before, we have
\begin{equation}\label{falphasum}
\eff(\alpha,T)=\frac{1}{\sqrt{\alpha}}\sum_{n=1}^{\infty}\frac{a_n}{\sqrt{n}}e^{-n/X}\cdot I_n+O(T^{\frac{2}{3}+\varepsilon}),
\end{equation}
for any $\varepsilon>0$, and where
\[
I_n=\int_{\alpha T}^{2\alpha T} 
e^{it\log\frac{t}{2\pi ne\alpha}-i\frac{\pi}{4}}dt.
\]
The integral $I_n$ is estimated by means of Lemmas 4.2, 4.6 and 4.4 of Titchmarsh \cite{Titch86} which
yield:
\begin{equation}\label{statphase}
\int_{\alpha T}^{2\alpha T} 
e^{it\log\frac{t}{2\pi ne\alpha}-i\frac{\pi}{4}}dt=\begin{cases} O(1)\,\,\text{if}\,\,2\pi n\ge 3T\\
2\pi \sqrt{n\alpha}e(-n\alpha)+\varphi(T)\,\,\text{if}\,\, T\le 2\pi n\le 2T,\,\,\text{and}\\
\varphi(T)\,\,\text{otherwise},\end{cases}
\end{equation}
where 
\[
\varphi(T)=O\left[T^{\frac{2}{5}}+\min \left(\sqrt{T},\frac{1}{\log (T/2\pi n)}\right)+\min\left(\sqrt{T},\frac{1}{\log (T/\pi n)}\right)\right].
\]
When $n\ge 3T$, using the first case of \eqref{statphase} together with \eqref{averageram} yields
\begin{equation}\label{bigt}
\frac{1}{\sqrt{\alpha}}\sum_{n\ge 3T}^{T^{1+\delta}}
\frac{a_n}{\sqrt{n}}e^{-n/X}\cdot I_n=O(T^{\frac{2}{3}+\varepsilon}).
\end{equation}
From this point onwards, we need to further modify the arguments of \cite{Sound05}.
If  $2\pi n$ lies in one of the intervals $P_{1,T}=[1,T-T^{3/4}]$, $P_{2,T}=[T+T^{3/4}, 2T-T^{3/4}]$ or 
$P_{3,T}=[2T+T^{3/4},3T]$, we see that $\varphi(T)=O(T^{2/5})$.
Hence, using \eqref{averageram} we see that 
\begin{equation}\label{smalltbigt}
\frac{1}{\sqrt{\alpha}}\sum_{2\pi n\in P_{j,T}}
\frac{a_n}{\sqrt{n}}e^{-n/X}\varphi(T)=O(T^{\frac{9}{10}+\varepsilon}),
\end{equation}
for $j\le i\le 3$.
When $2\pi n$ lies in either $Q_{1,T}=[T-T^{3/4},T+T^{3/4}]$ or in $Q_{2,T}=[2T-T^{3/4},2T+T^{3/4}]$, 
we have
\[
\min \left(\sqrt{T},\frac{1}{\log (T/2\pi n)}\right),
\min \left(\sqrt{T},\frac{1}{\log (T/\pi n)}\right)=O(\sqrt{T}).
\] 
It follows that $\varphi(T)=O(\sqrt{T})$ in these ranges. 
By \eqref{averageram}, we know that the sets of points 
$E_{j,\varepsilon}=\{U\in\R\,\vert\,\sum_{2\pi n\in Q_{j,U}}|a_n|>U^{\frac{3}{4}+\varepsilon}\}$,
$j=1,2$, have density zero as a subset of $\R$ for any $\varepsilon>0$, that
is, $\lim_{X\to\infty}\mu(E_{\varepsilon_j}\cap [0,X])/X=0$, where
$\mu$ is the Lebesgue measure on $\R$.
Thus, for all $T$ outside of a set $S_{\varepsilon}=E_{1,\varepsilon}\cup E_{2,\varepsilon}$
of density $0$ in $\R$, we have 
\begin{equation}\label{midt}
\frac{1}{\sqrt{\alpha}}\sum_{2\pi n\in Q_{j,T}}\frac{a_n}{\sqrt{n}}e^{-n/X}\varphi(T)
=O(T^{\frac{3}{4}+\varepsilon})
\end{equation}
for $j=1,2$. It follows that if $T\not\in S_{\varepsilon}$, \eqref{falphasum}
\eqref{bigt}, \eqref{smalltbigt} and \eqref{midt} show that
\[
\eff(\alpha,T)=2\pi\cdot
\sum_{T<2\pi n<2T}a_ne^{-2\pi in\alpha}+O(T^{\frac{9}{10}+\varepsilon})
\]
for any $\varepsilon>0$. Hence,
\[
\eff(\alpha)=\lim_{T\to\infty}\frac{\eff(\alpha,T)}{T^{1+iA}}=
\lim_{T\to\infty}\frac{1}{T^{1+iA}}\cdot 2\pi\cdot
\sum_{T<2\pi n<2T}a_ne^{-2\pi in\alpha}
\]
is periodic with period $1$, where the limit $T\to \infty$ is taken in 
$\R\setminus S_{\varepsilon}$. We substitute 
$\alpha +1$ in  \eqref{falphaeqn} to conclude that $C\pi Q^2=q$ must be a positive 
integer, and that $a_nn^{-iA}$ is periodic with period $q$. This proves the first assertion of the first 
part of the theorem. 

Once the periodicity of $a_nn^{-iA}$ has been established, the passage to 
the second assertion of the first part of the theorem made in \eqref{lincomb} is quite short and easy.
Since these arguments are identical to those of the proof of Theorem 8.1 of \cite{KaPe99},
we do not repeat them here. We note that the formulation in \cite{KaPe99} is
actually slightly sharper with a more precise description of the Dirichlet
polynomials $P_{\chi}$.

If we further assume that $F(s)$ satisfies (P4) or the stronger
(P4'), the second and third assertions of the theorem
follow follow almost immediately (see \cite{Sound05}, for instance).
\end{proof}

\section{Primitivity of cuspidal $L$-functions of $GL_2/\Q$}

Recall that an element $F(s)$ of $\lsel$ is called primitive if $F(s)=F_1(s)F_2(s)$
implies that either $F_1(s)$ or $F_2(s)$ is a unit. We say that an element
of $\lsel$ is almost primitive if $F(s)=F_1(s)F_2(s)$ implies that either $d_{F_1}=0$ or $d_{F_2}=0$. 
Using the third part of 
Theorem \ref{lselzero} together 
with the theorem above, we obtain the following corollary by induction on the degree.
\begin{corollary}\label{lselfactors} Every element of $\lsel$ (resp. $\gsel$) factors
into a product of primitive elements.
\end{corollary}
\begin{proof}
By using induction on the degree, we see from Theorem \ref{lselzerooneempty}
that factorization into into a product of primitive elements and elements of degree $0$
holds in $\gsel$ and $\lsel$. In \cite{KaPe03} the elements
of $\esel_0=\gsel_0=\lsel_0$, are shown to factorise into primitives, whence the 
proof.
\end{proof}
Since $\gselp_0=\lselp_0^{\rm{f}}=\{1\}$, an even easier version of the proof above gives
\begin{corollary}\label{lselpfactors}
Every element of $\gselp$ (resp. $\lselp^{\rm{f}}$) factors into a product of primitive elements in $\gselp$
(resp. $\lselp^{\rm{f}}$).
\end{corollary}

Let $f$ be a cuspidal eigenform on the upper half-plane, and let $L(s,f)$ be its associated
$L$-function. Recall that $f$ is either a holomorphic cusp form or Maass cusp form 
(which is real analytic). In the language of representation theory, $L(s,f)$ is the $L$-function associated
to the cuspidal automorphic representation $\pi_f$ of $\gltwo(\aq)$ attached to $f$.
We will assume that $L(s,f)$ is normalised so that it satisfies a functional
equation of the form \eqref{fnaleqn}. 
It is well known that $L(s,f)$ lies in $\gselp$ - this result is classical. When
$f$ is a holomorphic form, Deligne's celebrated proof of the Ramanujan conjecture
shows that $L(s,f)$ lies in $\sel$. 
Using Theorem \ref{mainthm} we can prove
the following.
\begin{theorem}\label{gltwoprimitive} The function $L(s,f)$ is primitive in $\gselp$.
\end{theorem}
\begin{proof}
By Theorem \ref{lselpfactors} we know that $L(s,f)$ must factor into a product of 
primitive elements. Because of Theorem \ref{lselzerooneempty} we know that only the
following two types of factorizations are possible. Either
\[
L(s,f)=F_0(s)F_2(s),
\]
with $F_0(s)$ of degree zero and $F_2(s)$ primitive of degree $2$, or
\[
L(s,f)=F_0(s)F_1(s)F_2(s),
\]
with $F_0(s)$ of degree $0$, and $F_1(s)$ and $F_2(s)$ both primitive of degree $1$. In either
case, we know that $F_0(s)=1$ by Theorem \ref{lselzero}.

In the first case there is nothing to prove. In the second case, the second 
part of Theorem \ref{mainthm} shows that $F_1(s)=L(s+it_1,\chi_1)$ and 
$F_2(s)=L(s+it_2,\chi_2)$ for Dirichlet characters $\chi_1$ and $\chi_2$ and
real numbers $t_1$ and $t_2$. Hence, we get
\begin{equation}\label{gltwofactors}
L(s,f)=L(s+it_1,\chi_1)L(s+it_2,\chi_2).
\end{equation}
This can be seen to be impossible as follows. For any Dirichlet series
$F(s)=\sum_{n=1}^{\infty}a_nn^{-s}$, we define the twist $F(s,\chi)$ of $F(s)$ by a 
Dirichlet character $\chi$ by 
\[
F(s,\chi)=\sum_{n=1}^{\infty}\frac{\chi(n)a_n}{n^s}.
\]
We denote the twist of $L(s,\chi_i)$ by $\chi$ by $L(s,\chi_i\chi)$ and the twist
of $L(s,f)$ by $\chi$ by $L(s,f\times\chi)$. One sees easily that if 
$F(s)=F_1(s)F_2(s)$ then $F(s,\chi)=F_1(s,\chi)F_2(s,\chi)$.
Let $S$ denote a finite set of places of $\Q$ containing all the primes
dividing the conductors of $f$, $\chi_1$ and $\chi_2$.
Twisting both sides of \eqref{gltwofactors} by $\bar{\chi}_1$ outside of $S$, we get
\[
L_S(s,f\times \bar{\chi}_1)=L_S(s+it_1,\chi_1\bar{\chi}_1)L_S(s+it_2,\chi_2\bar{\chi}_1),
\]
where $F_S(s)=\prod_{p\not\in S}F_p(s)$ for elements of $\gselp$.
The left-hand side is holomorphic on the line $\res=1$ by the classical work of
Hecke. On the right-hand side, 
$L_S(s+it_1,\chi_1\bar{\chi}_1)$ has a simple pole at $s=1-it_1$, while 
$L_S(s+it_2,\chi_2\bar{\chi}_1)$ is non-vanishing there (it may also have a simple
pole there if $\chi_1=\chi_2$). Thus the right-hand side is not holomorphic on the 
line $\res=1$, giving a contradiction.
\end{proof}
Since primitivity in $\gselp$ {\it a fortiori} implies primitivity in $\sel$, we can recover the
following result of Kaczorowski and Perelli as a corollary.
\begin{corollary} If $L(s,f)$ lies in $\sel$, it is primitive in $\sel$.
\end{corollary}

\section{Comparing zeros of $L$-functions} We return to a theme taken up first in
\cite{Ragh99} and more recently in \cite{Booker2015}.
Since the $L$-functions that are of interest in this section arise as
$L$-functions associated to automorphic representations, they come
naturally equipped with an Euler product. We will thus work in the 
class $\lselp(1)$.

If $L_1(s)\ne L_2(s)$ are elements in $\lselp(1)$ we would 
like to conclude (in many cases) that $L_1(s)/L_2(s)$ has infinitely many
poles, that is, that there are infinitely many zeros (counted with multiplicity) of 
$L_2(s)$ that are not zeros of $L_1(s)$. Our results in this paper
for $\lselp(1)$ allow us to consider several examples which 
were not covered by the results in \cite{Booker2015}.

As in the previous section $F(s,\chi)$ will denote the twist of the Dirichlet
series $F(s)$ by a Dirichlet character $\chi$, and 
$F_S(s)=\prod_{p\not\in S}F_p(s)$ for elements of $\lselp$ and $S$ a finite
set of primes.

\begin{theorem} \label{zerothm} Suppose $F_j(s)$ $j=1,2$ are elements of $\lselp(1)$ and
assume that $F_1(s)\ne F_2(s)$. Let $S$ be any finite set of primes and suppose that
\begin{enumerate}
\item $F_{1,S}(s,\chi)$ is holomorphic on $\re(s)=1$, and 
\item $F_{2,S}(s,\chi)$ is non-vanishing on $\re(s)=1$ 
\end{enumerate}
for every primitive Dirichlet character $\chi$.
If $F(s)=F_1(s)/F_2(s)$ with $d_F\in [0,1]$, then
$F(s)$ must have infinitely many poles.
\end{theorem}
\begin{proof} The arguments are similar to those in \cite{Ragh99} but
we now have the more powerful Theorem \ref{mainthm}.
If $F(s)$ has only finitely many poles, it must lie in $\lsel(1)$, and, in fact,
in $\lselp(1)$. It follows from Theorem \ref{mainthm} that the coefficients of
$F(s)$ are periodic with some period $q\in \N$. We let $S$ be the set of primes
dividing $q$. If $d_F=1$, we know by the second 
assertion of Theorem \ref{mainthm} that 
\[
F_{1,S}(s)=F_{2,S}(s)L_S(s+iA,\chi_0)
\]
for some Dirichlet character $\chi_0\Mod{q}$. 
We can assume that $S$ includes all the places where
$\chi_0$ is ramified, since the equality above holds for any larger set of primes containing
$S$. Twisting both sides of the equation above by $\chi_0^{-1}$, we see that
\[
F_{1,S}(s,\chi_0^{-1})=F_{2,S}(s,\chi_0^{-1})\zeta_S(s+iA).
\]

But $\zeta_S(s+iA)$ has a simple pole at $s=1-iA$, while  
$F_2(s,\chi_0^{-1})$, and hence, $F_{2,S}(s,\chi_0^{-1})$ is
non-vanishing there (by hypothesis). Thus, the right-hand side of the equation 
above has a simple pole at $s=1-iA$. On the other hand,
$F_1(s,\chi_0^{-1})$, and hence,
$F_{1,S}(s,\chi_0^{-1})$, is holomorphic at $s=1-iA$, yielding a contradiction.

If $d_F=0$, we apply our proof above to $J(s)=F(s)L(s,\chi)$, 
where $\chi$ is a primitive Dirichlet character such that $L(s,\chi)\ne F_2(s)$.
Thus $J(s)$ has degree $1$ and satisfies all the hypothesis of theorem, so 
by our proof above we get the stronger result that $J(s)$ has infinitely many poles.

If $0<d_F<1$, it follows immediately from Theorem \ref{lselzerooneempty} that the result is
vacuously true.
\end{proof}
\begin{remark} \label{criticalstrip}  Since we have absolute convergence in $\res>1$ for both $F_1(s)$ and
$F_2(s)$ and non-vanishing on $\res=1$ for $F_2(s)$, we see that the infinitely many 
poles of $F(s)$ lie in the critical strip. Thus, the infinitely many poles of 
$F(s)$ do not arise because of the trivial zeros of $F_2(s)$.
\end{remark}
\begin{remark}
The proof shows that we do not actually require that $F_i(s)$, $i=1,2$, individually belong to $\lselp$. We require
only that the quotient does. Thus, the theorem above applies to Artin $L$-functions
which have not yet been proven to lie in $\lselp$, but for which the functional equations
and relevant holomorphy and non-vanishing results for character twists are known.
\end{remark}
\begin{remark} The proof shows that we require the holomorphy and the non-vanishing only
for the {\it incomplete} twisted $L$-functions. This is usually easier to obtain
in practice. In fact, it is enough to show these properties for a fixed finite set $S$ which
contains all the primes dividing $N_1/N_2$, where the $N_j$ are the conductors of the 
$F_j(s)$, $j=1,2$.
\end{remark} 
\begin{remark} In \cite{Ragh99}, a similar theorem was proved for $d_F=0,1$, but essentially assuming
that the gamma factors at infinity were the same for $F_1(s)$ and $F_2(s)$, since the only 
classification theorems available at the time assumed $r^{\prime}=0$. This was of course
a strong restriction. There were also strong restrictions on the conductors.
\end{remark}
We apply our theorem to the following pair of functions. 
Let $\pi_i$, $1\le i\le 4$ be (unitary) cuspidal automorphic representations of 
$\glniak$ respectively. 
We take the tensor product $L$-functions $F_1(s)=L(s,\pi_1\times\pi_2)$ 
and $F_2(s)=L(s,\pi_3\times\pi_4)$. A series of papers due to
Jacquet-Piatetski-Shapiro-Shalika \cite{JaSh811,JaSh812,JPSS83}, as well as 
Shahidi \cite{Shahidi81,Shahidi88, Shahidi90} and Moeglin-Waldspurger \cite{MoWa89}
show that $F_j(s)\in \lsel(1)$, $j=1,2$, while the relevant non-vanishing statements for 
character twists are due to Shahidi \cite{Shahidi81}. The boundedness of the 
L-functions in vertical strips was proved in \cite{GeSh01}.
It follows that $F_j(s)$, $j=1,2$ satisfy all 
the conditions of the theorem.

One expects that if $\pi_i$ and $\pi_j$ ($1\le i,j\le 2$) are all 
distinct, then $F_1(s)\ne F_2(s)$ almost always. However, there are exceptions, 
and in practice, it is extremely
difficult to rule out the possibility that $F_1(s)=F_2(s)$. If $\pi_1\simeq\pi_3$ and
$L(s,\pi_2)\ne L(s,\pi_4)$, then we can show that $F_1(s)\ne F_2(s)$. 
Let us now assume further that $n_2=n_4$. In this case the quotient 
$F(s)$ has degree $0$ and satisfies all the hypotheses of Theorem \ref{zerothm}. It 
follows that $F(s)$ has infinitely many poles, that is, there are infinitely many zeros
(counted with multiplicty) of $L(s,\pi_1\times\pi_4)$ which are not zeros of 
$L(s,\pi_1\times\pi_2)$. In view of Remark \ref{criticalstrip}, these poles lie in
the critical strip. When $\pi_1$ and $\pi_3$
are chosen to the the trivial representation of $\glone(\ak)$, we recover the
theorem for the standard $L$-functions of $\glnak$. We can thus record the 
following corollary to Theorem \ref{zerothm}.
\begin{corollary} Let $\pi_1$, $\pi_2$ and $\pi_4$ be (unitary) cuspidal automorphic
representations of $\glniak$ for $i=1,2,4$. Assume that $n_2=n_4$ and that 
$L(s,\pi_2)\ne L(s,\pi_4)$. Then $L(s,\pi_1\times\pi_2)/L(s,\pi_1\times \pi_4)$
has infinitely many poles in the critical strip $0<\res<1$. 
\end{corollary}
The point about the example in this corollary is that the functions $F_j(s)$, $j=1,2$ are not known
to belong to $\bselp$, and thus Theorem 1.7 of \cite{Booker2015} could not have been applied in this case. 
We give two more examples below outside the purview of Booker's results.

Let $\pi$ be a (unitary) cuspidal automorphic representation of $\glnak$. The work of Shahidi
and Kim-Shahidi (see \cite{Shahidi81,Shahidi88, Shahidi90, Shahidi97} and \cite{Kim99}) shows that  the 
symmetric and exterior square $L$-functions, 
$L(s,\vee^2(\pi))$ and $L(s,\wedge^2 (\pi))$, lie in 
$\lselp(1)$. The relevant holomorphy and non-vanishing results for twists are also known by 
\cite{Shahidi97}, and the boundedness in vertical strips by \cite{GeSh01}, 
so our theorem applies to quotients of these $L$-functions (and
quotients of products of these $L$-functions) as well. 
Again, these $L$-functions are not known to lie in $\bselp$ and thus give more
examples where Theorem \ref{zerothm} where yields new results.

\bibliographystyle{alpha}
\bibliography{../../../../Bibtex/master2020}

\end{document}